\documentclass[a4paper,12pt]{amsart}
\usepackage{amsmath}
\usepackage{amssymb}
\usepackage{amsthm,amsxtra}
\usepackage{cite}
\usepackage{enumerate}
\usepackage{enumitem}
\usepackage[mathscr]{eucal}
\usepackage{adjustbox}
\usepackage{float}
\usepackage{graphicx}
\usepackage{mathtools}
\usepackage{url}
\usepackage{cleveref}
\usepackage{tikz-cd}
\usepackage{tkz-graph}
\usepackage{tkz-berge}
\usetikzlibrary{positioning,arrows,shapes.geometric,trees,decorations.text}
\usepackage{geometry}
\newgeometry{vmargin={40mm}, hmargin={30mm,30mm}}

\newtheorem{Def}{Definition}[section]
\newtheorem{Th}[Def]{Theorem}
\newtheorem{Ex}[Def]{Example}
\newtheorem{Lemma}[Def]{Lemma}
\newtheorem{Prop}[Def]{Proposition}
\newtheorem{Cor}[Def]{Corollary}
\newtheorem{Rem}[Def]{Remark}

{}
\newtheorem{Prob}[Def]{Problem}

\DeclareMathOperator{\fin}{Fin}
\DeclareMathOperator{\maxfin}{maxfin}
\DeclareMathOperator{\mach}{MA+\neg CH}

\begin{document}
\title[On a variation of selective separability: S-separability]{On a variation of selective separability: S-separability}

\author[ D. Chandra, N. Alam and D. Roy ]{ Debraj Chandra$^*$, Nur Alam$^\dag$ and Dipika Roy$^*$ }
\newcommand{\acr}{\newline\indent}
\address{\llap{*\,}Department of Mathematics, University of Gour Banga, Malda-732103, West Bengal, India}
\email{debrajchandra1986@gmail.com, roydipika1993@gmail.com}
\address{\llap{\dag\,}Department of Mathematics, Directorate of Open and Distance Learning (DODL), University of Kalyani, Kalyani, Nadia-741235, West Bengal, India}
\email{nurrejwana@gmail.com}


\maketitle

\begin{abstract}
A space $X$ is M-separable (selectively separable) (Scheepers, 1999; Bella et al., 2009) if for every sequence $(Y_n)$ of dense subspaces of $X$ there exists a sequence $(F_n)$ such that for each $n$ $F_n$ is a finite subset of $Y_n$ and $\cup_{n\in \mathbb{N}} F_n$ is dense in $X$. In this paper, we introduce and study a strengthening of M-separability situated between H- and M-separability, which we call S-separability: for every sequence $(Y_n)$ of dense subspaces of $X$ there exists a sequence $(F_n)$ such that for each $n$ $F_n$ is a finite subset of $Y_n$ and for each finite family $\mathcal F$ of nonempty open sets of $X$ some $n$ satisfies $U\cap F_n\neq\emptyset$ for all $U\in \mathcal F$.
\end{abstract}

\noindent{\bf\keywordsname{}:} {Selective separable, M-separable, S-separable, H-separable, L-separable.}

\section{Introduction}
In \cite{coc6}, Scheepers introduced and investigated a selective version of separability called selective separability (also known as M-separability; see \cite{bella09}). This line of research was further developed in \cite{bella08,bella09,barman11}. Certain variations of selective separability was also introduced and studied in \cite{bella09}. Motivated by this, in this paper, we introduce a variation of selective separability, called S-separability. Note that S-separability lies between H-separability and M-separability.

We establish the basic theory of S-separability. It is preserved by dense and open subspaces, and by open continuous and closed irreducible images, but not by arbitrary continuous mappings. Products behave subtly: the product of two S-separable spaces need not be S-separable; nevertheless, if $X$ and $Y$ are countable H-separable spaces and $\pi w(Y)<\mathfrak b$, then $X\times Y$ is S-separable, and more generally, $\prod_{n\in \mathbb{N}}X_n$ is S-separable provided each finite initial product is S-separable. For compact spaces we obtain a complete characterization: a compact $X$ is S-separable if and only if $\pi w(X)=\omega$. Consequences include permanence results for continuous images of separable compact spaces that are scattered or of countable tightness, and an example of a compact S-separable space with a non–S-separable continuous image.

We link S-separability to selection principles in function spaces. For Tychonoff $X$ the following are equivalent: $C_p(X)$ is S-separable; $C_p(X)$ is M-separable; $iw(X)=\omega$ and every finite power of $X$ is Menger. For zero-dimensional $X$ the same equivalence holds for $C_p(X,\mathbb Q)$, $C_p(X,\mathbb Z)$, and $C_p(X,2)$. On the level of small subspaces, the threshold $\mathfrak d$ is sharp: the smallest $\pi$-weight of a countable non-S-separable space is $\mathfrak d$. In particular, every countable subspace of $2^\kappa$ is S-separable for $\kappa<\mathfrak d$, whereas $2^{\mathfrak d}$ contains a countable dense non-S-separable subspace. We also show (in ZFC) that $2^{\omega_1}$ contains a dense and hence a countable dense S-separable subspace. Finally, we define L-separability and compare it with separability, hereditary separability, and M-separability.

Several natural problems remain open, including whether in ZFC there exists a S-separable space that is not H-separable, and whether there exists a M-separable space that is not S-separable.

\section{Preliminaries}
By a space we always mean a topological space. All spaces are assumed to be Tychonoff; otherwise will be mentioned. For undefined notions and terminologies see \cite{Engelking}. $w(X)$ denotes the weight of a space $X$. A collection $\mathcal{B}$ of open sets of $X$ is called $\pi$-base (respectively, $\pi$-base at $x\in X$) if every nonempty open set in $X$ (respectively, every neighbourhood of $x$ in $X$) contains a nonempty member of $\mathcal{B}$. $\pi w(X) = \min \{|\mathcal{B}| : \mathcal{B} \text{ is a } \pi\text{-base for } X\}$ denotes the $\pi$-weight of $X$. The minimum cardinality of a $\pi$-base at $x\in X$ is denoted by $\pi_\chi(x, X)$ and $\pi_\chi(X) = \sup\{\pi_\chi(x, X) : x\in X\}$. $d(X)$ denotes the minimal cardinality of a dense subspace of $X$ and $\delta(X) = \sup \{d(Y) : Y \text{ is dense in }X\}$. For any space $X$, $d(X) \leq \delta(X) \leq \pi w(X) \leq w(X)$. The $i$-weight of a space $(X, \tau)$ is defined as \[iw((X, \tau)) = \min \{\kappa : \text{there is a Tychonoff topology } \tau^\prime \subseteq \tau \text{ such that } w((X, \tau^\prime)) = \kappa\}.\]

A space $X$ has countable fan tightness \cite{arhan92} if for any $x\in X$ and any sequence $(Y_n)$ of subsets of $X$ with $x\in \cap_{n\in \mathbb{N}} \overline{Y_n}$ there exists a sequence $(F_n)$ such that for each $n$ $F_n$ is a finite subset of $Y_n$ and $x \in \overline{\cup_{n\in \mathbb{N}} F_n}$. $X$ has countable fan tightness with respect to dense subspaces \cite{bella09} if for any $x\in X$ and any sequence $(Y_n)$ of dense subspaces of $X$ there exists a sequence $(F_n)$ such that for each $n$ $F_n$ is a finite subset of $Y_n$ and $x \in \overline{\cup_{n\in \mathbb{N}} F_n}$. A space $X$ has countable tightness (which is denoted by $t(X) = \omega$) if for each $x\in X$ and each $Y \subseteq X$ with $x\in \overline{Y}$ there exists a countable set $E \subseteq Y$ such that $x\in \overline{E}$. A space $X$ is scattered if every nonempty subspace $Y$ of $X$ has an isolated point. For a Tychonoff space $X$, $\beta X$ denotes the Stone-\v{C}ech compactification of $X$.

Let $X$ be a Tychonoff space and $C(X)$ be the set of all continuous real valued functions. As usual $C_p(X)$ denotes the space $C(X)$ with pointwise convergence topology. Let $f \in C(X)$. Then a basic open set of $f$ in $C_p(X)$ is of the form \[B(f, F, \epsilon) = \{g\in C(X) : |f(x) - g(x)| < \epsilon\; \forall x\in F\},\] where $F$ is a finite subset of $X$ and $\epsilon > 0$. For a separable metrizable space $X$, $C_p(X)$ is hereditarily separable. Also recall from \cite[Theorem II.2.10]{arhan92} that $C_p(X)$ is Menger if and only if $X$ is finite.

A space $X$ is said to be selective separable (M-separable) \cite{bella08} (see also \cite{bella09}) if for every sequence $(Y_n)$ of dense subspaces of $X$ there exists a sequence $(F_n)$ such that for each $n$ $F_n$ is a finite subset of $Y_n$ and $\cup_{n\in \mathbb{N}} F_n$ is dense in $X$. A space $X$ is said to be H-separable \cite{bella09} if for every sequence $(Y_n)$ of dense subspaces of $X$ there exists a sequence $(F_n)$ such that for each $n$ $F_n$ is a finite subset of $Y_n$ and every nonempty open set of $X$ intersects $F_n$ for all but finitely many $n$.

\section{Main Results}
\subsection{S-separability and basic observations on it}
We begin with the following definition of a class of spaces, also considered by Aurichi et al. \cite{aurichi25}.

\begin{Def}
\label{def1}
A space $X$ is said to be S-separable if for every sequence $(Y_n)$ of dense subspaces of $X$ there exists a sequence $(F_n)$ such that for each $n$ $F_n$ is a finite subset of $Y_n$ and for each finite collection $\mathcal{F}$ of nonempty open sets of $X$ there exists a $n$ such that $U\cap F_n \neq \emptyset$ for all $U\in \mathcal{F}$.
\end{Def}

Clearly every H-separable space is S-separable and every S-separable space is M-separable. In the following example we observe that the class of H-separable spaces is strictly contained in the class of S-separable spaces.
\begin{Ex}
\label{ex3}
There exists a S-separable space which is not H-separable.
\end{Ex}
\begin{proof}
Assume that $\mathfrak{b} < \mathfrak{d}$. By \cite[Theorem 31]{bella09}, there exists a countable dense subspace $Y$ of $2^\mathfrak{b}$ such that $Y$ is not H-separable. Also by Corollary~\ref{cor1}, every countable subspace of $2^\mathfrak{b}$ is S-separable. Thus $Y$ is S-separable.
\end{proof}

The above example is consistent with ZFC. So the following question can be asked.
\begin{Prob}
\label{prob3}
In ZFC, does there exist a S-separable space which is not H-separable?
\end{Prob}

The following question also remains open.
\begin{Prob}
\label{prob4}
Does there exist a M-separable space which is not S-separable?
\end{Prob}

\begin{Prop}
\label{prop3}
Let $X$ be a S-separable space. Then
\begin{enumerate}[wide=0pt,label={\upshape(\arabic*)},leftmargin=*,ref={\theProp(\arabic*)}]
  \item\label{prop301} every dense subspace of $X$ is S-separable.
  \item\label{prop302} every open subspace of $X$ is S-separable.
\end{enumerate}
\end{Prop}
\begin{proof}
$(1)$. Let $Y$ be a dense subspace of $X$. Pick a sequence $(Y_n)$ of dense subspaces of $Y$. Obviously each $Y_n$ is dense in $X$ because $Y$ is dense in $X$. Since $X$ is S-separable, there exists a sequence $(F_n)$ such that for each $n$ $F_n$ is a finite subset of $Y_n$ and for each finite collection $\mathcal{F}$ of nonempty open sets of $X$ there exists a $n\in \mathbb{N}$ such that $U\cap F_n \neq \emptyset$ for all $U\in \mathcal{F}$. Observe that the sequence $(F_n)$ witnesses for the sequence $(Y_n)$ that $Y$ is S-separable.

$(2)$. Let $Y$ be an open subspace of $X$. Pick a sequence $(Y_n)$ of dense subspaces of $Y$. Obviously for each $n$ $Z_n = (X\setminus Y) \cup Y_n$ is dense in $X$. Since $X$ is S-separable, there exists a sequence $(F_n)$ such that for each $n$ $F_n$ is a finite subset of $Z_n$ and for each finite collection $\mathcal{F}$ of nonempty open sets of $X$ there exists a $n\in \mathbb{N}$ such that $U\cap F_n \neq \emptyset$ for all $U\in \mathcal{F}$. Now we can express each $F_n$ as $F_n = A_n \cup B_n$, where $A_n$ is a finite subset of $X\setminus Y$ and $B_n$ is a finite subset of $Y_n$. Observe that the sequence $(B_n)$ witnesses for $(Y_n)$ that $Y$ is S-separable. Let $\mathcal{F}$ be a finite collection of nonempty open sets of $Y$. Since $Y$ is open in $X$, $\mathcal{F}$ is also a finite collection of nonempty open sets of $X$. Then we get a $m \in \mathbb{N}$ such that $U\cap F_m \neq \emptyset$ for all $U\in \mathcal{F}$. Since $U\subseteq Y$ and $A_m \subseteq X\setminus Y$, $U\cap B_m \neq \emptyset$. Thus $Y$ is S-separable.
\end{proof}

\begin{Prop}
\label{prop2}
For a space $X$ the following assertions are equivalent.
\begin{enumerate}[wide=0pt,label={\upshape(\arabic*)},leftmargin=*]
  \item $X$ is hereditarily S-separable.
  \item $X$ is hereditarily separable and every countable subspace of $X$ is S-separable.
\end{enumerate}
\end{Prop}
\begin{proof}
$(2)\Rightarrow (1)$. Let $Y$ be a subspace of $X$. Pick a sequence $(Y_n)$ of dense subspaces of $Y$. Since $X$ is hereditarily separable, we get a sequence $(Z_n)$ such that for each $n$ $Z_n$ is a countable dense subspace of $Y_n$. Choose $Z= \cup_{n\in \mathbb{N}} Z_n$. Clearly for each $n$ $Z_n$ is dense in $Z$ and $Z$ is S-separable as every countable subspace of $X$ is S-separable. Apply the S-separability of $Z$ to $(Z_n)$ to obtain a sequence $(F_n)$ such that for each $n$ $F_n$ is a finite subset of $Z_n$ and for every finite collection $\mathcal{F}$ of nonempty open sets of $Z$ there exists a $n\in \mathbb{N}$ such that $U\cap F_n\neq \emptyset$ for all $U\in \mathcal{F}$. We claim that the sequence $(F_n)$ guarantees for $(Y_n)$ that $Y$ is S-separable. Let $\mathcal{G}$ be a finite collection of nonempty open sets of $Y$. Then $\mathcal{F} = \{U\cap Z : U\in \mathcal{G}\}$ is a finite collection of nonempty open sets of $Z$ as $Z$ is dense in $Y$. This gives a $k\in \mathbb{N}$ such that $V\cap F_k\neq \emptyset$ for all $V\in \mathcal{F}$, i.e. $U\cap F_k\neq \emptyset$ for all $U\in \mathcal{G}$. Thus $Y$ is S-separable.
\end{proof}

\begin{Prop}
\label{prop5}
If a space $X$ has an open dense S-separable subspace, then $X$ is S-separable.
\end{Prop}
\begin{proof}
Let $Y$ be an open dense S-separable subspace of $X$. Pick a sequence $(Y_n)$ of dense subspaces of $X$. For each $n$ $Z_n = Y_n\cap Y$ is dense in $Y$. Since $Y$ is S-separable, there exists a sequence $(F_n)$ such that for each $n$ $F_n$ is a finite subset of $Z_n$ and for every finite collection $\mathcal{F}$ of nonempty open sets of $Y$ there exists a $n\in \mathbb{N}$ such that $U\cap F_n \neq \emptyset$ for all $U\in \mathcal{F}$. Observe that $(F_n)$ witnesses for $(Y_n)$ that $X$ is S-separable. Let $\mathcal{F}$ be a finite collection of nonempty open sets of $X$. Define a finite collection $\mathcal{G}$ of nonempty open sets of $Y$ as $\mathcal{G} = \{U\cap Y : U\in \mathcal{F}\}$. Then there exists a $m\in \mathbb{N}$ such that $V\cap F_m \neq \emptyset$ for all $V\in \mathcal{G}$, i.e. $U\cap Y\cap F_m \neq \emptyset$ for all $U\in \mathcal{F}$. Thus $U\cap F_m \neq \emptyset$ for all $U\in \mathcal{F}$ and hence $X$ is S-separable.
\end{proof}

Recall that a continuous mapping $f: X \to Y$ is irreducible if the only closed subset $C$ of $X$ satisfying $f(C) = Y$ is $C = X$.
\begin{Prop}
\label{prop4}
Let $X$ be a S-separable space. Then
\begin{enumerate}[wide=0pt,label={\upshape(\arabic*)},leftmargin=*,ref={\theProp(\arabic*)}]
\item\label{prop401} every open continuous image of $X$ is S-separable.
\item\label{prop402} every closed irreducible image of $X$ is S-separable.
\end{enumerate}
\end{Prop}
\begin{proof}
$(1)$. Let $f: X \to Y$ be an open continuous mapping from $X$ onto a space $Y$. Let $(Y_n)$ be a sequence of dense subspaces of $Y$. Now $(f^{-1}(Y_n))$ is a sequence of dense subspaces of $X$ as dense subspaces are inverse invariant under open mappings. Since $X$ is S-separable, there exists a sequence $(F_n)$ such that for each $n$ $F_n$ is a finite subset of $f^{-1}(Y_n)$ and for every finite collection $\mathcal{F}$ of nonempty open sets of $X$ there exists a $n\in \mathbb{N}$ such that $U\cap F_n \neq \emptyset$ for all $U \in \mathcal{F}$. For each $n$ choose $A_n = f(F_n)$. Observe that the sequence $(A_n)$ witnesses for $(Y_n)$ that $Y$ is S-separable. Let $\mathcal{F}$ be a finite collection of nonempty open sets of $Y$. Then $f^{-1}(\mathcal{F}) = \{f^{-1}(V) : V\in \mathcal{F}\}$ is a finite collection of nonempty open sets of $X$. This gives a $m\in \mathbb{N}$ such that $f^{-1}(V) \cap F_m \neq \emptyset$ for all $V\in \mathcal{F}$. Clearly $V\cap A_m \neq \emptyset$ for all $V\in \mathcal{F}$. Thus $Y$ is S-separable.

$(2)$. Let $f: X \to Y$ be a closed irreducible mapping from $X$ onto a space $Y$. Let $(Y_n)$ be a sequence of dense subspaces of $Y$. Now $(f^{-1}(Y_n))$ is a sequence of dense subspaces of $X$ as dense subspaces are inverse invariant under closed irreducible mappings. Since $X$ is S-separable, there exists a sequence $(F_n)$ such that for each $n$ $F_n$ is a finite subset of $f^{-1}(Y_n)$ and for every finite collection $\mathcal{F}$ of nonempty open sets of $X$ there exists a $n\in \mathbb{N}$ such that $U\cap F_n \neq \emptyset$ for all $U \in \mathcal{F}$. For each $n$ choose $A_n = f(F_n)$. Observe that the sequence $(A_n)$ witnesses for $(Y_n)$ that $Y$ is S-separable. Let $\mathcal{F}$ be a finite collection of nonempty open sets of $Y$. Then $f^{-1}(\mathcal{F}) = \{f^{-1}(V) : V\in \mathcal{F}\}$ is a finite collection of nonempty open sets of $X$. This gives a $m\in \mathbb{N}$ such that $f^{-1}(V) \cap F_m \neq \emptyset$ for all $V\in \mathcal{F}$. Clearly $V\cap A_m \neq \emptyset$ for all $V\in \mathcal{F}$. Thus $Y$ is S-separable.
\end{proof}

\begin{Th}
\label{thm12}
If $X$ is a countable H-separable space and $Y$ is a countable space with $\pi w(Y) < \mathfrak{b}$, then $X\times Y$ is H-separable.
\end{Th}
\begin{proof}
Let $\kappa < \mathfrak{b}$ and $\mathcal{B} = \{B_\alpha : \alpha < \kappa\}$ be a $\pi$-base of $Y$. Pick a sequence $(Y_n)$ of dense subspaces of $X\times Y$. For each $n$ we can choose $Y_n = \{x_m^{(n)} : m\in \mathbb{N}\}$. Let $p_1: X\times Y\to X$ and $p_2: X\times Y\to Y$ be the projection mappings on $X$ and $Y$ respectively. Let $\alpha < \kappa$ be fixed. For each $n$ let $Z_n^{(\alpha)} = p_1(Y_n\cap (X\times B_\alpha))$. Clearly each $Z_n$ is dense in $X$. Since $X$ is H-separable, there exists a sequence $(F_n^{(\alpha)})$ such that for each $n$ $F_n^{(\alpha)}$ is a finite subset of $Z_n^{(\alpha)}$ and every nonempty open set of $X$ intersects $F_n^{(\alpha)}$ for all but finitely many $n$. For each $n$ choose a finite subset $G_n^{(\alpha)}$ of $Y_n\cap (X\times B_\alpha)$ such that $F_n^{(\alpha)} = p_1(G_n^{(\alpha)})$. Choose a $f_\alpha \in \mathbb{N}^\mathbb{N}$ such that $G_n^{(\alpha)} \subseteq \{x_m^{(n)} : m \leq f_\alpha(n)\}$. Thus we obtain a subset $\{f_\alpha : \alpha < \kappa\}$ of $\mathbb{N}^\mathbb{N}$. Since $\kappa < \mathfrak{b}$, there exists a $f\in \mathbb{N}^\mathbb{N}$ such that $f_\alpha \leq^* f$ for all $\alpha < \kappa$. For each $n$ let $F_n = \{x_m^{(n)} : m\leq f(n)\}$. Observe that the sequence $(F_n)$ witnesses for $(Y_n)$ that $X\times Y$ is H-separable. Let $U$ be a nonempty open set of $X\times Y$. Choose a nonempty open set $V$ of $X$ and a $\alpha < \kappa$ such that $V \times B_\alpha \subseteq U$. Then $V$ intersects $F_n^{(\alpha)}$ for all but finitely many $n$. We can find a $n_0\in \mathbb{N}$ such that $V\cap F_n^{(\alpha)} \neq \emptyset$ and $f_\alpha(n) \leq f(n)$ for all $n\geq n_0$. We claim that $(V \times B_\alpha) \cap F_n \neq \emptyset$ for all $n\geq n_0$. Choose a $n\geq n_0$ and a $x \in V\cap F_n^{(\alpha)}$. Now $x\in F_n^{(\alpha)}$ gives $x = p_1(x, y)$ for some $(x, y) \in G_n^{(\alpha)}$. It follows that $y\in B_\alpha$ and $(x, y) = x_m^{(n)}$ for some $m \leq f_\alpha(n)$. Since $f_\alpha(n) \leq f(n)$, from the construction of $F_n$ we can say that $(x , y) \in F_n$. Also since $(x, y) \in V\times B_\alpha$, $(V \times B_\alpha) \cap F_n \neq \emptyset$. Thus $(V \times B_\alpha) \cap F_n \neq \emptyset$ for all $n\geq n_0$, i.e. $U \cap F_n \neq \emptyset$ for all $n\geq n_0$. Hence $X\times Y$ is H-separable.
\end{proof}

\begin{Cor}
\label{cor14}
If $X$ is a countable H-separable space and $Y$ is a countable space with $\pi w(Y) < \mathfrak{b}$, then $X\times Y$ is S-separable.
\end{Cor}

From \cite[Corollary 2.5]{babin09} we get two separable metrizable spaces $X$ and $Y$ such that $C_p(X)$ and $C_p(Y)$ are M-separable but the product space $C_p(X) \times C_p(Y)$ is not M-separable (hence not S-separable). By Theorem~\ref{thm7}, $C_p(X)$ and $C_p(Y)$ are S-separable. Thus we can say that the product of two S-separable spaces need not be S-separable. However, we obtain the following result.

\begin{Th}
\label{thm13}
Let $\{X_n : n\in \mathbb{N}\}$ be a family of spaces and $X = \prod_{n\in \mathbb{N}} X_n$. If each $n$ $X_n^\prime = \prod_{k=1}^n X_k$ is S-separable, then $X$ is S-separable.
\end{Th}
\begin{proof}
Let $(Y_n)$ be a sequence of dense subspaces of $X$. Let $\{N_n : n\in \mathbb{N}\}$ be a partition of $\mathbb{N}$ into disjoint infinite subsets. For each $n$ let $p_n : X\to X_n^\prime$ be the natural projection. Fix $n$. Apply the S-separability of $X_n^\prime$ to $(p_n(Y_k) : k\in N_n)$ to obtain a sequence $(F_k : k\in N_n)$ such that for each $k\in N_n$, $F_k$ is a finite subset of $Y_k$ and for every finite collection $\mathcal{F}$ of nonempty open sets of $X_n^\prime$ there exists a $k\in N_n$ such that $U\cap p_n(F_k) \neq \emptyset$ for all $U\in \mathcal{F}$. Thus we obtain a sequence $(F_n)$ such that for each $n$ $F_n$ is a finite subset of $Y_n$. Observe that $(F_n)$ witnesses for $(Y_n)$ that $X$ is S-separable. Let $\mathcal{F}$ be a finite collection of nonempty open sets of $X$. For each $U\in \mathcal{F}$ choose a basic member $V(U)$ of $X$ such that $V(U) \subseteq U$. Then for each $U\in \mathcal{F}$, $V(U) = \prod_{n\in \mathbb{N}} V(U)_n$ with $V(U)_n = X_n$ for all but finitely many $n$. We can find a $m\in \mathbb{N}$ such that $V(U)_n = X_n$ for all $n\geq m$ and for all $U\in \mathcal{F}$. There exists a $k\in N_m$ such that $p_m(V(U)) \cap p_m(F_k) \neq \emptyset$, i.e. $\prod_{i=1}^m V(U)_i \cap p_m(F_k) \neq \emptyset$ for all $U\in \mathcal{F}$. Now $\prod_{i=1}^m V(U)_i \cap p_m(F_k) \neq \emptyset$ gives a $x(U) = (x(U)_n) \in F_k$ such that $(x(U)_i : 1\leq i \leq m) \in \prod_{i=1}^m V(U)_i$. Since $V(U)_n = X_n$ for all $n\geq m$, $x(U)\in V(U)$. Thus $V(U)\cap F_k \neq \emptyset$ for all $U\in \mathcal{F}$. Hence $X$ is S-separable.
\end{proof}

\begin{Th}[{\cite{juhas89}}]
\label{thm9}
If $X$ is any compact Hausdorff space, then $X$ has a dense subspace $Y$ with $d(Y) = \pi w(X)$.
\end{Th}

\begin{Prop}
\label{prop6}
A compact space $X$ is S-separable if and only if $\pi w(X) = \omega$.
\end{Prop}
\begin{proof}
Assume that $X$ is S-separable. By Theorem~\ref{thm9}, there exists a dense subspace $Y$ of $X$ such that $d(Y) = \pi w(X)$. Then by Proposition~\ref{prop301}, $Y$ is S-separable and so $Y$ is separable, i.e. $d(Y) = \omega$. Thus $\pi w(X) = \omega$.

Conversely suppose that $\pi w(X) = \omega$. This gives $\delta (X) = \omega$. By Theorem~\ref{thm1}, $X$ is S-separable.
\end{proof}

It is important to mention that compactness is essential in the above result. Indeed, the space $C_p([0,1])$ is S-separable because $iw([0,1]) = \omega$ and every finite power of $[0,1]$ is Menger (see Theorem~\ref{thm7}). But $\pi w(C_p([0,1])) = \mathfrak{c}$.

\begin{Prop}
\label{prop7}
Let $X$ be a separable compact space. If $X$ is either scattered or has countable tightness, then every continuous image of $X$ is S-separable.
\end{Prop}
\begin{proof}
Let $Y$ be a continuous image of $X$. If $X$ is scattered, then $Y$ is scattered and separable. It follows that $\pi w(Y) = \omega$ and hence by Proposition~\ref{prop6}, $Y$ is S-separable.

Now assume that $X$ has countable tightness, i.e. $t(X) = \omega$. This gives $t(Y) = \omega$ (see \cite[1.1.1]{arhan78}) and consequently $\pi_\chi(Y) = \omega$ (see \cite{sapir75}). From the separability of $Y$ we can say that $\pi w(Y) = \omega$. Again by Proposition~\ref{prop6}, $Y$ is S-separable.
\end{proof}

\begin{Cor}
\label{cor10}
If a compact space $X$ is either scattered or has countable tightness, then every separable subspace of $X$ is S-separable.
\end{Cor}
\begin{proof}
Suppose that $X$ is scattered. Let $Y$ be a separable subspace of $X$. Since every subspace of a scattered space is scattered, $Y$ is scattered. It follows that $\pi w(Y) = \omega$ and hence by Proposition~\ref{prop6}, $Y$ is S-separable.

Now suppose that $X$ has countable tightness, i.e. $t(X) = \omega$. This gives $t(Y) = \omega$ and consequently $\pi_\chi(Y) = \omega$ (see \cite{sapir75}). From the separability of $Y$ we can say that $\pi w(Y) = \omega$. Again by Proposition~\ref{prop6}, $Y$ is S-separable.
\end{proof}

Thus if $X$ is a countable non-S-separable space, then $X$ cannot be embedded in a compact space of countable tightness.

In the following example we show that a continuous image of a compact S-separable space may not be S-separable.
\begin{Ex}
\label{ex2}
There exists a compact S-separable space which has a non-S-separable continuous image.
\end{Ex}
\begin{proof}
Consider the space $\beta \mathbb{N}$. Since $\pi w(\beta \mathbb{N}) = \omega$ (see \cite{kkjv}), by Proposition~\ref{prop6}, $\beta \mathbb{N}$ is S-separable. The Tychonoff cube $\mathbb{I}^\mathfrak{c}$ can be obtained as a continuous image of $\beta \mathbb{N}$, where $\mathbb{I} = [0,1]$. $\mathbb{I}^\mathfrak{c}$ has a countable dense subspace which is not S-separable (see \cite[Example 2.14]{bella08} for explanation). By Proposition~\ref{prop301}, $\mathbb{I}^\mathfrak{c}$ is not S-separable.
\end{proof}

Note that every continuous image of a compact space $X$ is S-separable does not guarantee that $t(X) = \omega$. Indeed, there exists a separable compact scattered space $X$ with $t(X) \geq \omega_1$ (see \cite[Example 2.8]{bella08}). By Proposition~\ref{prop7}, every continuous image of $X$ is S-separable.

\subsection{Further observations on S-separability}
For a set $Y\subseteq\mathbb{N}^\mathbb{N}$, $\maxfin(Y)$ is defined as \[\maxfin(Y)=\left\{\max\{f_1,f_2,\dotsc,f_k\} : f_1,f_2,\dotsc,f_k\in Y\;\text{and}\; k\in\mathbb{N}\right\},\] where $\max\{f_1,f_2,\dotsc,f_k\}(n)=\max\{f_1(n),f_2(n),\dotsc,f_k(n)\}$ for all $n\in\mathbb{N}$.
\begin{Th}
\label{thm1}
If $\delta(X) = \omega$ and $\pi w(X) < \mathfrak{d}$, then $X$ is S-separable.
\end{Th}
\begin{proof}
Let $\mathcal{B} = \{U_\alpha : \alpha < \kappa\}$ be a $\pi$-base for $X$ with $U_\alpha \neq \emptyset$ for all $\alpha < \kappa$. Assume that $\kappa < \mathfrak{d}$. Let $(Y_n)$ be a sequence of dense subspaces of $X$. Since $\delta(X) = \omega$, we get a sequence $(Z_n)$ such that for each $n$ $Z_n = \{x_m^{(n)} : m\in \mathbb{N}\}$ is a dense subspace of $Y_n$. For each $\alpha< \kappa$ we define a $f_\alpha \in \mathbb{N}^\mathbb{N}$ by $f_\alpha(n) = \min \{m\in \mathbb{N} : x_m^{(n)} \in U_\alpha\}$. Since $Y = \{f_\alpha : \alpha < \kappa\}$ has cardinality less than $\mathfrak{d}$, $\maxfin(Y)$ also has cardinality less than $\mathfrak{d}$. Then there exist a $g\in \mathbb{N}^\mathbb{N}$ and for each finite $F \subseteq \kappa$ a $n_F \in \mathbb{N}$ such that $f_F (n_F) < g(n_F)$ with $f_F \in \maxfin(Y)$. We use the convention that if $F=\{\alpha\}$, $\alpha < \kappa$, then we write $f_\alpha$ instead of $f_F$. For each $n$ let $F_n = \{x_m^{(n)} : m\leq g(n)\}$. Observe that the sequence $(F_n)$ witnesses for $(Y_n)$ that $X$ is S-separable. Let $\mathcal{F}$ be a finite collection of nonempty open sets of $X$. Then there exists a finite $F \subseteq \kappa$ such that each member of $\mathcal{F}$ contains a $U_\alpha$ for some $\alpha \in F$. We claim that $U\cap F_{n_F} \neq \emptyset$ for all $U\in \mathcal{F}$. Pick a $U\in \mathcal{F}$. Choose a $\beta \in F$ such that $U_\beta \subseteq U$. Now from the definition of $f_\beta$ we get $x_{f_\beta(n_F)}^{(n_F)} \in U_\beta$. Since $f_F(n_F) < g(n_F)$ and $f_F = \max\{f_\alpha : \alpha \in F\}$, $f_\beta(n_F) < g(n_F)$. It follows that $x_{f_\beta(n_F)}^{(n_F)} \in F_{n_F}$ and hence $U_\beta \cap F_{n_F} \neq \emptyset$, i.e. $U \cap F_{n_F} \neq \emptyset$. Thus $X$ is S-separable.
\end{proof}

Recall from \cite{kkjv} that for $2^\kappa$ (the Cantor cube of weight $\kappa$), $\pi w(2^\kappa) = \kappa = w(2^\kappa)$.
\begin{Cor}
\label{cor1}
If $\kappa < \mathfrak{d}$, then every countable subspace of $2^\kappa$ is S-separable.
\end{Cor}
\begin{proof}
Let $X$ be a countable subspace of $2^\kappa$. Then $\delta(X) = \omega$. Since $w(2^\kappa) = \kappa$ and weight is monotone (i.e. for every subspace $Y$ of a space $X$, $w(Y) \leq w(X)$), $w(X) \leq \kappa$. Also since $\pi w (X) \leq w(X)$, $\pi w(X) \leq \kappa$. Hence the result.
\end{proof}

\begin{Cor}
\label{cor11}
Assume $\mach$. Then every countable subspace of $2^{\omega_1}$ is S-separable.
\end{Cor}
\begin{proof}
Note that MA gives $\mathfrak{d} = \mathfrak{c}$ and so $\omega_1 < \mathfrak{d}$. Let $X$ be a countable subspace of $2^{\omega_1}$. Then $\delta(X) = \omega$. Since $w(2^{\omega_1}) = \omega_1$, $w(X) \leq \omega_1$. Also since $\pi w(X) \leq w(X)$, $\pi w(X) \leq \omega_1 < \mathfrak{d}$. Thus $X$ is S-separable.
\end{proof}

\begin{Th}[{\cite[Theorem 2.18]{bella08}}]
\label{thm2}
The space $2^\mathfrak{d}$ contains a countable dense subspace which is not M-separable (hence not S-separable).
\end{Th}

\begin{Cor}
\label{cor13}
The space $2^\mathfrak{d}$ contains a countable dense subspace which is not S-separable.
\end{Cor}

\begin{Cor}
\label{cor2}
The smallest $\pi$-weight of a countable non-S-separable space is $\mathfrak{d}$.
\end{Cor}
\begin{proof}
Let $X$ be a countable space. Then $\delta(X) = \omega$. If $\pi w(X) < \mathfrak{d}$, then by Theorem~\ref{thm1}, $X$ is S-separable.

Now by Corollary~\ref{cor13}, $2^\mathfrak{d}$ contains a countable subspace $Y$ such that $\pi w(Y) = \mathfrak{d}$ and $Y$ is not S-separable. Thus the smallest $\pi$-weight of a countable non-S-separable space is $\mathfrak{d}$.
\end{proof}

Let $F$ be a subset of a space $X$. For each $x\in F$ choose an open set $U_x$ of $X$ containing $x$. By a neighbourhood of $F$ we mean the collection $\{U_x : x\in F\}$ and we use $\mathcal{N}(F)$ to denote a neighbourhood of $F$. We now introduce the following notions.
\begin{enumerate}[wide=0pt,leftmargin=*]
\item A space $X$ is said to satisfy $(*_1)$ if for any finite $F\subseteq X$ and any sequence $(Y_n)$ of subsets of $X$ with $F\subseteq \cap_{n\in \mathbb{N}} \overline{Y_n}$ there exists a sequence $(F_n)$ such that for each $n$ $F_n$ is a finite subset of $Y_n$ and for each neighbourhood $\mathcal{N}(F)$ of $F$ there exists a $n\in \mathbb{N}$ such that $U\cap F_n \neq \emptyset$ for all $U\in \mathcal{N}(F)$.
\item A space $X$ is said to satisfy $(*_1)$ with respect to dense subspaces if for any finite $F\subseteq X$ and any sequence $(Y_n)$ of dense subspaces of $X$ there exists a sequence $(F_n)$ such that for each $n$ $F_n$ is a finite subset of $Y_n$ and for each neighbourhood $\mathcal{N}(F)$ of $F$ there exists a $n\in \mathbb{N}$ such that $U\cap F_n \neq \emptyset$ for all $U\in \mathcal{N}(F)$.
\end{enumerate}

It is immediate that if $X$ satisfies $(*_1)$ (respectively, $(*_1)$ with respect to dense subspaces), then $X$ has countable fan tightness (respectively, countable fan tightness with respect to dense subspaces). Also if $X$ satisfies $(*_1)$, then it satisfies $(*_1)$ with respect to dense subspaces, and if $X$ is S-separable, then it satisfies $(*_1)$ with respect to dense subspaces.

\begin{Prop}
\label{prop1}
A separable space $X$ is S-separable if and only if $X$ satisfies $(*_1)$ with respect to dense subspaces.
\end{Prop}
\begin{proof}
Suppose that $X$ satisfies $(*_1)$ with respect to dense subspaces. Let $(Y_n)$ be a sequence of dense subspaces of $X$. Let $\{N_n : n\in \mathbb{N}\}$ be a partition of $\mathbb{N}$ into pairwise disjoint infinite subsets. Since $X$ is separable, we have a countable dense subspace $Y$ of $X$. Choose $\fin (Y) = \{G_n : n\in \mathbb{N}\}$. Fix $n\in \mathbb{N}$. Apply the property $(*_1)$ with respect to dense subspaces of $X$ to $(Y_m : m\in N_n)$ to obtain a sequence $(F_m : m\in N_n)$ such that for each $m\in N_n$, $F_m$ is a finite subset of $Y_m$ and for each neighbourhood $\mathcal{N}(G_n)$ of $G_n$ there exists a $m\in N_n$ such that $U\cap F_m \neq \emptyset$ for all $U\in \mathcal{N}(G_n)$. We claim that the sequence $(F_n)$ witnesses for $(Y_n)$ that $X$ is S-separable. Let $\mathcal{F}$ be a finite collection of nonempty open sets of $X$. For each $U \in \mathcal{F}$ we can pick a $x_U \in U \cap Y$ and let $G$ be the collection of all such $x_U$. Then $G = G_k$ for some $k\in \mathbb{N}$ and $\mathcal{F}$ is a neighbourhood of $G_k$. This gives a $m \in N_k$ such that $U \cap F_m \neq \emptyset$ for all $U\in \mathcal{F}$. Thus $X$ is S-separable.

The other direction is trivial.
\end{proof}

\begin{Lemma}
\label{lemma101}
If $X$ satisfies $(*_1)$, then every subspace of $X$ also satisfies $(*_1)$.
\end{Lemma}
\begin{proof}
Let $Y$ be a subspace of $X$. Let $F$ be a finite subset of $Y$ and $(Y_n)$ be sequence of subsets of $Y$ such that $F\subseteq \overline{Y_n}^Y$ for all $n$. Clearly $F\subseteq \overline{Y_n}$ for all $n$. Since $X$ satisfies $(*_1)$, there exists a sequence $(F_n)$ such that for each $n$ $F_n$ is a finite subset of $Y_n$ and for every neighbourhood of $\mathcal{N}(F)$ of $F$ in $X$ there exists a $n\in \mathbb{N}$ such that $U\cap F_n\neq \emptyset$ for all $U\in \mathcal{N}(F)$. Observe that the sequence $(F_n)$ witnesses for $F$ and $(Y_n)$ that $Y$ satisfies $(*_1)$. Let $\mathcal{N}(F)$ be a neighbourhood of $F$ in $Y$. Then for each $V\in \mathcal{N}(F)$ we can find an open $U(V)$ set in $X$ such that $V = U(V) \cap Y$. Then $\{U(V) : V\in \mathcal{N}(F)\}$ is a neighbourhood of $F$ in $X$. This gives a $m\in \mathbb{N}$ such that $U(V) \cap F_m \neq \emptyset$ for all $V \in \mathcal{N}(F)$. It follows that $V\cap F_m \neq \emptyset$ for all $V\in \mathcal{N}(F)$. Thus $Y$ satisfies $(*_1)$.
\end{proof}

\begin{Th}[{\cite{arhan86,bella09},\cite[II.2.2. Theorem]{arhan92}}]
\label{thm3}
For a space $X$ the following assertions are equivalent.
\begin{enumerate}[wide=0pt,label={\upshape(\arabic*)},leftmargin=*]
  \item $C_p(X)$ has countable fan tightness.
  \item $C_p(X)$ has countable fan tightness with respect to dense subspaces.
  \item Every finite power of $X$ is Menger.
\end{enumerate}
\end{Th}

\begin{Th}
\label{thm4}
For a space $X$ the following assertions are equivalent.
\begin{enumerate}[wide=0pt,label={\upshape(\arabic*)},leftmargin=*]
  \item $C_p(X)$ satisfies $(*_1)$.
  \item $C_p(X)$ satisfies $(*_1)$ with respect to dense subspaces.
  \item Every finite power of $X$ is Menger.
  \item Every finite power of $X$ is Scheepers.
\end{enumerate}
\end{Th}
\begin{proof}
$(2) \Rightarrow (3)$. If $C_p(X)$ satisfies $(*_1)$ with respect to dense subspaces, then $C_p(X)$ has countable fan tightness with respect to dense subspaces. By Theorem~\ref{thm3}, every finite power of $X$ is Menger.

$(3) \Rightarrow (4)$. Let $k\in \mathbb{N}$ and $Y = X^k$. Since every finite power of $X$ is Menger, every finite power of $Y$ is also Menger. It follows that $Y$ is Scheepers (see \cite{coc2}). Thus every finite power of $X$ is Scheepers.

$(4) \Rightarrow (1)$. We will follow the technique of the proof of \cite[Theorem II.2.2]{arhan92} to prove this implication. Let $F$ be a finite subset of $C_p(X)$ and $(Y_n)$ be a sequence of subsets of $C_p(X)$ such that $F\subseteq \cap_{n\in \mathbb{N}} \overline{Y_n}$. Let $k \in \mathbb{N}$ be fixed. Pick a $f\in F$. Since $f\in \cap_{n\in \mathbb{N}} \overline{Y_n}$, for each $n\in \mathbb{N}$ and each $x = (x_1, x_2, \dotsc, x_k) \in X^k$ there exists a $g_{f,x}^{(n,k)} \in Y_n$ such that $|g_{f,x}^{(n,k)}(x_i) - f(x_i)| < \frac{1}{k}$ for all $1\leq i \leq k$. Since $g_{f,x}^{(n,k)}$ and $f$ are continuous, for each $1 \leq i \leq k$ we can choose an open set $U_{f,i,x}^{(n,k)}$ of $X$ containing $x_i$ such that $|g_{f,x}^{(n,k)}(y) - f(y)| < \frac{1}{k}$ for all $y \in U_{f,i,x}^{(n,k)}$. Then $U_{f,x}^{(n,k)} = U_{f,1,x}^{(n,k)}\times U_{f,2,x}^{(n,k)}\times \dotsc \times U_{f,k,x}^{(n,k)}$ is an open set in $X^k$ containing $x$ and $|g_{f,x}^{(n,k)}(y_i) - f(y_i)| < \frac{1}{k}$ for all $y = (y_1, y_1, \dotsc, y_k) \in U_{f,x}^{(n,k)}$. Choose $U_x^{(n,k)} = \cap_{f\in F} U_{f,x}^{(n,k)}$. Clearly $|g_{f,x}^{(n,k)}(y_i) - f(y_i)| < \frac{1}{k}$ for all $y = (y_1, y_1, \dotsc, y_k) \in U_x^{(n,k)}$ and for all $f\in F$. For each $n \in \mathbb{N}$, $\mathcal{U}_n^{(k)} = \{U_x^{(n,k)} : x\in X^k\}$ is an open cover of $X^k$. Apply the Scheepers property of $X^k$ to $(\mathcal{U}_n^{(k)})$ to obtain a sequence $(\mathcal{V}_n^{(k)})$ such that for each $n$ $\mathcal{V}_n^{(k)}$ is a finite subset of $\mathcal{U}_n^{(k)}$ and $\{\cup \mathcal{V}_n^{(k)} : n\in \mathbb{N}\}$ is an $\omega$-cover of $X^k$. Then we can obtain a sequence $(A_n^{(k)})$ of finite subsets of $X^k$ such that for each $n$ $\mathcal{V}_n^{(k)} = \{U_x^{(n,k)} : x\in A_n^{(k)}\}$. For each $n\in \mathbb{N}$ and each $f\in F$, $F_{f,n}^{(k)} = \{g_{f,x}^{(n,k)} : x\in A_n^{(k)}\}$ is a finite subset of $Y_n$. For each $n\in \mathbb{N}$ and each $f\in F$ let $F_{f,n} = \cup_{k\leq n} F_{f,n}^{(k)}$. Also for each $n$ let $F_n = \cup_{f\in F} F_{f,n}$. Thus we define a sequence $(F_n)$ such that for each $n$ $F_n$ is a finite subset of $Y_n$. Observe that the sequence $(F_n)$ witnesses for $F$ and $(Y_n)$ that $C_p(X)$ satisfies $(*_1)$. Let $\mathcal{N}(F) = \{U_f : f\in F\}$ be a neighbourhood of $F$ in $C_p(X)$. Then we get a finite subset $F^\prime = \{z_1, z_2, \dotsc, z_m\}$ of $X$ and an $\epsilon > 0$ such that $B(f, F^\prime, \epsilon) \subseteq U_f$ for all $f\in F$. We can assume that $\frac{1}{m} < \epsilon$. Since $z = (z_1, z_2, \dotsc, z_m) \in X^m$, there exists a $p \geq m$ such that $z\in \cup \mathcal{V}_p^{(m)}$. If possible suppose that there exists a $f\in F$ such that $B(f, F^\prime, \epsilon) \cap F_{f,p} = \emptyset$. It follows that $B(f, F^\prime, \epsilon) \cap F_{f,p}^{(i)} = \emptyset$ for all $i \leq p$, i.e. $B(f, F^\prime, \epsilon) \cap F_{f,p}^{(m)} = \emptyset$. Subsequently $g_{f,x}^{(p,m)} \notin B(f, F^\prime, \epsilon)$ for all $x \in A_{p}^{(m)}$. Then for each $x\in A_p^{(m)}$, $|g_{f,x}^{(p,m)}(z_{i_x}) - f(z_{i_x})| \geq \epsilon > \frac{1}{m}$ for some $1 \leq i_x \leq m$. Consequently $z \notin U_x^{(p,m)}$ for all $x \in A_p^{(m)}$, i.e. $z \notin \cup \mathcal{V}_p^{(m)}$. Which is absurd. Thus $B(f, F^\prime, \epsilon) \cap F_{f,p}\neq \emptyset$ and so $B(f, F^\prime, \epsilon) \cap F_p\neq \emptyset$. It follows that $U_f \cap F_p \neq \emptyset$ for all $f\in F$. Hence $C_p(X)$ satisfies $(*_1)$.
\end{proof}

Combining Theorems~\labelcref{thm3,thm4} we obtain the following result.
\begin{Th}
\label{thm5}
For a space $X$ the following assertions are equivalent.
\begin{enumerate}[wide=0pt,label={\upshape(\arabic*)},leftmargin=*]
  \item $C_p(X)$ satisfies $(*_1)$.
  \item $C_p(X)$ has countable fan tightness.
  \item $C_p(X)$ satisfies $(*_1)$ with respect to dense subspaces.
  \item $C_p(X)$ has countable fan tightness with respect to dense subspaces.
  \item Every finite power of $X$ is Menger.
  \item Every finite power of $X$ is Scheepers.
\end{enumerate}
\end{Th}

\begin{Cor}
\label{cor3}
If every finite power of $X$ is Menger, then every separable subspace of $C_p(X)$ is S-separable.
\end{Cor}
\begin{proof}
Let $Y$ be a separable subspace of $C_p(X)$. Suppose that every finite power of $X$ is Menger. Then by Theorem~\ref{thm5}, $C_p(X)$ satisfies $(*_1)$. Also by Lemma~\ref{lemma101}, $Y$ satisfies $(*_1)$. Thus by Proposition~\ref{prop1}, $Y$ is S-separable.
\end{proof}

From \cite[2.3.1]{arhan78} we now consider a space $V$ called countable Fr\'{e}chet-Urysohn fan. Since $V$ has a dense subspace of isolated points, $\pi w(V) = \omega$. By Theorem~\ref{thm1}, $V$ is S-separable. But $V$ does not have countable fan tightness, i.e. $V$ does not satisfy $(*_1)$. However, we observe in Theorem~\ref{thm7} that $C_p(X)$ is S-separable implies that $C_p(X)$ satisfies $(*_1)$. We first recall the following well known result which will be used subsequently.

\begin{Th}[{\cite{noble74} (see also \cite[Theorem I.1.5]{arhan92})}]
\label{thm6}
For a Tychonoff space $X$, $iw(X) = d(C_p(X))$.
\end{Th}

\begin{Th}
\label{thm7}
For a Tychonoff space $X$ the following assertions are equivalent.
\begin{enumerate}[wide=0pt,label={\upshape(\arabic*)},leftmargin=*]
  \item $C_p(X)$ is separable and satisfies $(*_1)$.
  \item $C_p(X)$ is separable and satisfies $(*_1)$ with respect to dense subspaces.
  \item $C_p(X)$ is S-separable.
  \item $C_p(X)$ is M-separable.
  \item $iw(X) = \omega$ and every finite power of $X$ is Menger.
\end{enumerate}
\end{Th}
\begin{proof}
The implications $(1) \Rightarrow (2)$ and $(3) \Rightarrow (4)$ are obvious. The implication $(4) \Rightarrow (5)$ was shown in \cite[Theorem 21]{bella09}.

$(2) \Rightarrow (3)$. By Theorem~\ref{thm4}, every finite power of $X$ is Menger. Then by Corollary~\ref{cor3}, $C_p(X)$ is S-separable.

$(5) \Rightarrow (1)$. By Theorem~\ref{thm5}, $C_p(X)$ satisfies $(*_1)$. Since $iw(X) = \omega$, by Theorem~\ref{thm6}, $d(C_p(X)) = \omega$. It follows that $C_p(X)$ is separable.
\end{proof}

\begin{Cor}
\label{cor4}
If $C_p(X)$ is S-separable, then for each $n \in \mathbb{N}$, $C_p(X^n)$ is S-separable, and hence $(C_p(X))^\omega$ is S-separable.
\end{Cor}
\begin{proof}
Let $n \in \mathbb{N}$ and $Y = X^n$. Since $C_p(X)$ is S-separable, $iw(X) = \omega$ and every finite power of $X$ is Menger. It follows that every finite power of $Y$ is also Menger and $iw(Y) = \omega$. Thus $C_p(Y)$ is S-separable. It follows that for each $n \in \mathbb{N}$, $C_p(X^n)$ is S-separable, and hence $(C_p(X))^\omega$ is S-separable.
\end{proof}

\begin{Cor}
\label{cor5}
If $X$ is a second countable space such that $C_p(X)$ is S-separable, then $C_p(X)$ is hereditarily S-separable.
\end{Cor}
\begin{proof}
Let $Y$ be a subspace of $C_p(X)$. Since $X$ is a second countable space, it is separable metrizable. It follows that $C_p(X)$ is hereditarily separable and hence $Y$ is separable. Since $C_p(X)$ is S-separable, $C_p(X)$ satisfies $(*_1)$. By Lemma~\ref{lemma101}, $Y$ satisfies $(*_1)$ and so does $(*_1)$ with respect to dense subspaces. Also by Proposition~\ref{prop1}, $Y$ is S-separable. Thus $C_p(X)$ is hereditarily S-separable.
\end{proof}

\begin{Cor}
\label{cor6}
If $X$ is a second countable space with $|X|< \mathfrak{d}$, then $C_p(X)$ is hereditarily S-separable.
\end{Cor}
\begin{proof}
Let $Y$ be a subspace of $C_p(X)$. Since $X$ is a second countable space, it is separable metrizable. It follows that $C_p(X)$ is hereditarily separable and hence $Y$ is separable. Also every finite power of $X$ is second countable and so every finite power of $X$ is Lindel\"{o}f. Since $|X|< \mathfrak{d}$, every finite power of $X$ is Menger (see \cite{coc2}). It follows that $C_p(X)$ satisfies $(*_1)$ as $iw(X) = d(C_p(X)) = \omega$. Then $Y$ satisfies $(*_1)$ and so does $(*_2)$ with respect to dense subspaces. Also by Proposition~\ref{prop1}, $Y$ is S-separable. Thus $C_p(X)$ is hereditarily S-separable.
\end{proof}

Using Theorem~\ref{thm7} and \cite[Corollary 2.13]{bella08} the following result can be easily verified.
\begin{Prop}
\label{prop8}
The space $C_p(C_p(X))$ is S-separable if and only if $X$ is finite.
\end{Prop}

\begin{Rem}\rm
\label{rem1}
\hfill
\begin{enumerate}[wide=0pt,label={\upshape(\arabic*)},ref={\theRem(\arabic*)}]
\item\label{rem101} Recall that the Baire space $\mathbb{N}^\mathbb{N}$ can be condensed onto a compact space $X$. It follows that $C_p(X)$ can be densely embedded in $C_p(\mathbb{N}^\mathbb{N})$. Now $C_p(X)$ is separable, i.e. $d(C_p(X)) = \omega$ and so $iw(X) = \omega$ (see Theorem~\ref{thm6}). Since $X$ is compact, every finite power of $X$ is Menger and hence by Theorem~\ref{thm7}, $C_p(X)$ is S-separable. Let $Y$ be a countable dense subspace of $C_p(X)$. Then $Y$ is S-separable. Since $C_p(X)$ is densely embeddable in $C_p(\mathbb{N}^\mathbb{N})$, $Y$ is a countable dense S-separable subspace of $C_p(\mathbb{N}^\mathbb{N})$. Also since $C_p(\mathbb{N}^\mathbb{N})$ can be densely embedded in the Tychonoff cube $\mathbb{I}^\mathfrak{c}$, $Y$ is a countable dense S-separable subspace of $\mathbb{I}^\mathfrak{c}$. Thus $\mathbb{I}^\mathfrak{c}$ has a countable dense S-separable subspace.

\item\label{rem102} Moreover, we have already mentioned in Example~\ref{ex2} that $\mathbb{I}^\mathfrak{c}$ has a countable dense non-S-separable subspace.
\end{enumerate}
\end{Rem}

\begin{Th}
\label{thm8}
For a zero-dimensional space $X$ the following assertions are equivalent.
\begin{enumerate}[wide=0pt,label={\upshape(\arabic*)},leftmargin=*]
  \item $C_p(X, \mathbb{Q})$ is S-separable.
  \item $C_p(X, \mathbb{Z})$ is S-separable.
  \item $C_p(X, 2)$ is S-separable.
  \item $iw(X) = \omega$ and every finite power of $X$ is Menger.
\end{enumerate}
\end{Th}
\begin{proof}
$(1) \Rightarrow (2)$. For each $n \in \mathbb{Z}$ pick an irrational $x_n \in (n, n+1)$. Choose a $\psi \in \mathbb{Z}^\mathbb{Q}$ with $\psi(x) = n$ if $x\in (x_{n-1}, x_n)$. Let us define a mapping $\Phi : C_p(X, \mathbb{Q}) \to C_p(X, \mathbb{Z})$ by $\Phi(f) = \psi \circ f$ for all $f \in C_p(X, \mathbb{Q})$. Clearly $\Phi$ is open continuous from $C_p(X, \mathbb{Q})$ onto $C_p(X, \mathbb{Z})$. By Proposition~\ref{prop401}, $C_p(X, \mathbb{Z})$ is S-separable.

$(2) \Rightarrow (3)$. Pick a $\phi \in 2^\mathbb{Z}$ satisfying
\[\phi(n) = \begin{cases}
             0, & \mbox{if } n< 0 \\
             1, & \mbox{otherwise}.
            \end{cases}\]
Let us define a mapping $\Psi : C_p(X, \mathbb{Z}) \to C_p(X, 2)$ by $\Psi(f) = \phi \circ f$ for all $f \in C_p(X, \mathbb{Z})$. Now $\Psi$ is open continuous from $C_p(X, \mathbb{Z})$ onto $C_p(X, 2)$. By Proposition~\ref{prop401}, $C_p(X, 2)$ is S-separable.

$(3) \Rightarrow (4)$. Since $C_p(X, 2)$ is S-separable, it is M-separable. By \cite[Proposition 25]{bella09}, $iw(X) = \omega$ and every finite power of $X$ is Menger.

$(4) \Rightarrow (1)$. By Theorem~\ref{thm7}, $C_p(X)$ is S-separable. Since $X$ is zero-dimensional, $C_p(X, \mathbb{Q})$ is dense in $C_p(X)$. By Proposition~\ref{prop301}, $C_p(X, \mathbb{Q})$ is S-separable.
\end{proof}

With the help of \cite[Proposition 25]{bella09} the following corollary can be obtained.
\begin{Cor}
\label{cor9}
For a zero-dimensional space $X$ the following assertions are equivalent.
\begin{enumerate}[wide=0pt,label={\upshape(\arabic*)},leftmargin=*]
  \item $C_p(X, \mathbb{Q})$ is S-separable.
  \item $C_p(X, \mathbb{Q})$ is M-separable.
  \item $C_p(X, \mathbb{Z})$ is S-separable.
  \item $C_p(X, \mathbb{Z})$ is M-separable.
  \item $C_p(X, 2)$ is S-separable.
  \item $C_p(X, 2)$ is M-separable.
  \item $iw(X) = \omega$ and every finite power of $X$ is Menger.
\end{enumerate}
\end{Cor}

\begin{Cor}
\label{cor7}
Let $(X, \tau)$ be condensed onto a separable metrizable space $(X, \tau^\prime)$. If every finite power of $(X, \tau^\prime)$ is Menger, then there exists a dense S-separable subspace of $C_p(X, \tau^\prime)$.
\end{Cor}
\begin{proof}
We can think of $C_p(X, \tau^\prime)$ as subspace of $C_p(X, \tau)$. Clearly $C_p(X, \tau^\prime)$ is dense in $C_p(X, \tau)$. By Theorem~\ref{thm4}, $C_p(X, \tau^\prime)$ satisfies $(*_1)$ with respect to dense subspaces. Since $(X, \tau^\prime)$ is separable metrizable, $C_p(X, \tau^\prime)$ is hereditarily separable. It follows that $C_p(X, \tau^\prime)$ is S-separable (see Proposition~\ref{prop1}).
\end{proof}

In the following example we observe that separability plays a crucial role in the validity of the above result.
\begin{Ex}
\label{ex1}
There exists a metrizable space $X$ such that $C_p(X)$ contains a dense S-separable subspace but $C_p(X)$ is not S-separable.
\end{Ex}
\begin{proof}
Let $D(\mathfrak{c})$ be the discrete space of cardinality $\mathfrak{c}$. Now $C_p(D(\mathfrak{c})) = \mathbb{R}^\mathfrak{c}$. We interpret $\mathbb{R}^\mathfrak{c}$ as $\mathbb{R}^{2^\omega}$. Let $Y = C_p(2^\omega)$.
Then $Y$ is dense in $\mathbb{R}^{2^\omega}$. Note that $iw(2^\omega) = \omega$ and $2^\omega$ is compact, i.e. every finite power of $2^\omega$ is Menger. By Theorem~\ref{thm7}, $Y$ is S-separable. Thus $Y$ is a dense S-separable subspace of $C_p(D(\mathfrak{c})) = \mathbb{R}^\mathfrak{c}$. Again by Theorem~\ref{thm7}, $C_p(D(\mathfrak{c})) = \mathbb{R}^\mathfrak{c}$ is not S-separable as $D(\mathfrak{c})$ is not Menger.
\end{proof}

It is natural to ask the following question.
\begin{Prob}
\label{prob1}
If $C_p(X)$ contains a dense S-separable subspace, then can $X$ be condensed onto a second countable space every finite power of which is Menger?
\end{Prob}

\begin{Th}[{\cite[Theorem 46]{bella09}}]
\label{thm10}
The space $2^\mathfrak{c}$ contains a countable dense H-separable subspace.
\end{Th}

\begin{Cor}
\label{cor8}
The space $2^\mathfrak{c}$ contains a countable dense S-separable subspace.
\end{Cor}

Thus in ZFC, $2^\mathfrak{c}$ has a countable dense S-separable subspace. Under the assumption $\mach$, every countable subspace of $2^{\omega_1}$ is S-separable (see Corollary~\ref{cor11}). It is natural to ask the following question.

\begin{Prob}
\label{prob2}
In ZFC, does $2^{\omega_1}$ contain a countable dense S-separable subspace?
\end{Prob}

We now give an affirmative answer to the above problem.
\begin{Th}
\label{thm11}
The space $2^{\omega_1}$ contains a dense S-separable subspace.
\end{Th}
\begin{proof}
Assume that $\omega_1 < \mathfrak{b}$. Then by Corollary~\ref{cor1}, every countable subspace of $2^{\omega_1}$ is S-separable. Thus we can find a countable dense S-separable subspace of $2^{\omega_1}$.

Next assume that $\omega_1 = \mathfrak{b}$. From \cite[Theorem 10]{bartos06} we can say that there exists a zero-dimensional metrizable space $X$ with cardinality $\mathfrak{b}$ such that every finite power of $X$ is Menger.
Clearly $iw(X) = \omega$. By Theorem~\ref{thm8}, $C_p(X, 2)$ is a S-separable subspace of $2^{\omega_1}$. Also $C_p(X, 2)$ is dense in $2^{\omega_1}$.
\end{proof}

\begin{Cor}
\label{cor12}
The space $2^{\omega_1}$ contains a countable dense S-separable subspace.
\end{Cor}
\begin{proof}
By Theorem~\ref{thm11}, we get a dense S-separable subspace $Y$ of $2^{\omega_1}$. Let $Z$ be a countable dense subspace of $Y$. Then $Z$ is S-separable. Clearly $Z$ is dense in $2^{\omega_1}$.
\end{proof}

\subsection{L-separability}
\begin{Def}
\label{def2}
A space $X$ is said to be L-separable if for every dense subset $D$ of $X$ there exists a countable subset $C$ of $D$ such that $C$ is dense in $X$ (or equivalently, if every dense subspace of $X$ is separable).
\end{Def}

Verify that a space $X$ is L-separable if and only if $\delta(X) = \omega$. It is clear that every M-separable space is L-separable. However, the converse does not hold in general. Indeed, there exist L-separable spaces which are not M-separable. For instance, since every countable space is L-separable, any countable space that fails to be M-separable provides a counterexample. Such examples exist in ZFC (see Theorem~\ref{thm2}). Therefore, M-separability is a strictly stronger property than L-separability. However, since a space $X$ is L-separable if and only if $\delta(X) = \omega$, every L-separable space with $\pi$-weight less than $\mathfrak{d}$ is S-separable (and hence M-separable) (see Theorem~\ref{thm1}).

Also note that L-separability implies separability. We now observe that the converse is not true.

\begin{Ex}
\label{ex4}
There exists a Tychonoff separable space which is not L-separable.
\end{Ex}
\begin{proof}
Let $I$ be an index set with $|I|=\mathfrak{c}$, and consider the product $X=[0,1]^I$ with the product topology. Then $X$ is separable as the product of at most continuum many separable spaces is separable. Let
\[\Sigma(X) =\{x\in X:\operatorname{supp}(x)=\{i\in I: x(i)\neq 0\}\text{ is countable}\}.\]

We first show that $\Sigma(X)$ is dense in $X$. Let $U$ be a basic open cylinder determined by finitely many coordinates $F\subseteq I$. One can choose $y\in \Sigma(X)$ that matches the prescribed values on $F$ and is $0$ elsewhere, so $y\in U\cap \Sigma(X)$.

We claim that $\Sigma(X)$ is not separable. Indeed, fix any countable $A\subseteq \Sigma(X)$ and set $J = \bigcup_{a\in A} \operatorname{supp}(a)$, which is countable. Choose $j^\prime\in I\setminus J$ and a non-degenerate open interval $V\subseteq (0,1]$. Consider the cylinder $U=\{x\in X: x(j^\prime)\in V\}$. Then $W = U\cap \Sigma(X)$ is a nonempty open subset of the subspace $\Sigma(X)$. For every $a\in A$ we have $a(j^\prime)=0$, hence $A\cap U = \emptyset$. If $x\in W$, then the basic neighborhood of $x$ that restricts only the $j^\prime$-coordinate to lie in $V$ is disjoint from $A$; therefore $x\notin \overline{A}^{X}$, and so $x\notin \overline{A}^{\Sigma(X)} = \overline{A}^{X}\cap \Sigma(X)$. Thus $W\cap \overline{A}^{\Sigma(X)} = \emptyset$, showing that no countable $A\subseteq \Sigma(X)$ is dense in $\Sigma(X)$. Hence $\Sigma(X)$ is a dense non-separable subspace of the separable space $X$.

Since $X$ has a dense non-separable subspace $\Sigma(X)$, it follows that $X$ is not $L$-separable.
\end{proof}

Let $N(X) = (X\times\{0\})\cup (\mathbb{R}\times (0,\infty))$ be the Niemytzki plane on a set $X\subseteq \mathbb{R}$. The topology on $N(X)$ is defined as follows. $\mathbb{R}\times(0,\infty)$ has the Euclidean topology and the set $X\times\{0\}$ has the topology generated by all sets of the form $\{(x,0)\}\cup U$, where $x\in X$ and $U$ is an open disc in $\mathbb{R}\times (0,\infty)$ which is tangent to $X\times \{0\}$ at the point $(x,0)$. The topology on $N(X)$ is also called Niemytzki's tangent disk topology. It is to be noted that Niemytzki originally defined $N(\mathbb{R})$ (see \cite{Niemytzki}).

Recall that a space $X$ is called hereditarily separable if every subspace of $X$ is separable. Clearly if a space $X$ is hereditarily separable, then $X$ is $L$-separable. Using the following lemma we will show in Example~\ref{ex5} that the converse fails in general.

\begin{Lemma}
\label{lemma2}
If $X$ has an open dense subspace $Y$ that is second countable, then $X$ is S-separable.
\end{Lemma}
\begin{proof}
Let $(Y_n)$ be a sequence of dense subspaces of $X$. Since $Y$ is open and dense in $X$, each $Y_n\cap Y$ is dense in $Y$. Fix a countable base $\mathcal{B} = \{B_n : n\in\mathbb{N}\}$ for the subspace $Y$.

For each $n\in \mathbb{N}$ choose a finite set $F_n\subseteq Y_n\cap Y$ with $F_n\cap B_i\neq \emptyset$ for all $i=1, 2, \dotsc, n$ (possible since $Y_n\cap Y$ is dense in $Y$). We claim that $(F_n)$ witnesses S-separability. Let $\mathcal{F}$ be any finite family of nonempty open subsets of $X$. Since $Y$ is dense and open, $U\cap Y\neq \emptyset$ for each $U\in\mathcal{F}$, and $U\cap Y$ is open in $Y$. For each $U\in \mathcal{F}$ pick $B_{i(U)}\in \mathcal{B}$ with $B_{i(U)}\subseteq U\cap Y$. Let $n_0 = \max\{i(U) : U\in \mathcal{F}\}$. By construction, $F_{n_0}$ meets every $B_{i(U)}$, hence $F_{n_0}\cap U\neq \emptyset$ for all $U\in \mathcal{F}$. Thus $X$ is S-separable.
\end{proof}

\begin{Ex}
\label{ex5}
There exists a Tychonoff S-separable (hence L-separable) space which is not hereditarily separable.
\end{Ex}
\begin{proof}
Consider the Niemytzki plane $N(\mathbb{R})$ on $\mathbb{R}$. Choose $H = \mathbb{R}\times (0,\infty)$ and $L = \mathbb{R}\times \{0\}$. $N(\mathbb{R})$ is not hereditarily separable because $L$, being an uncountable discrete subspace of it, is not separable. On the other hand, $H$ is open and dense in $N(\mathbb{R})$, and the subspace topology on $H$ is the usual Euclidean topology, so $H$ is second countable. By Lemma~\ref{lemma2}, it follows that $N(\mathbb{R})$ is S-separable.
\end{proof}

For separable metrizable spaces, separability, hereditary separability, and L-separability coincide.

However, hereditary separability does not imply M-separability (see Theorem~\ref{thm2}). Also M-separability does not imply hereditary separability as the Niemytzki plane $N(\mathbb{R})$ on $\mathbb{R}$, being a separable Fr\'{e}chet space, is M-separable (see \cite[Theorem 2.9]{barman11}) but not hereditary separability (see Example~\ref{ex5}).

\end{document}